\documentclass{amsart}

\usepackage[left=1.2in,right=1.2in,top=1.2in,bottom=1.2in]{geometry}

\usepackage[utf8]{inputenc}
\usepackage[english]{babel}

\usepackage{graphicx}
\graphicspath{ {./images/} }
\usepackage{amsmath, lipsum}
\usepackage{amssymb}
\usepackage{amsthm}
\usepackage{framed}
\usepackage{mathtools}
\usepackage{xcolor}     % For color definitions
\usepackage[
colorlinks=true,
linkcolor=blue,     % Color for internal links (like table of contents)
citecolor=red,      % Color for citations
urlcolor=magenta    % Color for URLs
]{hyperref}
\usepackage{enumerate}
\usepackage{enumitem}
\usepackage{booktabs}
\usepackage{array}
\usepackage{faktor}
\usepackage{cancel}
\usepackage{makecell}
\usepackage{pbox}
\usepackage{bm}
\usepackage[toc,page]{appendix}
\usepackage{etoolbox}
\usepackage{stmaryrd}
\usepackage{extarrows}
\usepackage{blindtext}
\usepackage{comment}
\usepackage{centernot}
\usepackage{floatrow}
\allowdisplaybreaks[1]

\usepackage{tikz}
\usepackage{tikz-cd}
\usetikzlibrary{arrows.meta}
\usepackage[all,cmtip]{xy}

\newcommand{\got}{\mathfrak}
\newcommand{\cali}{\mathcal}

\newcommand{\eqdef}{\coloneqq}

\newcommand{\mto}{\mapsto}
\newcommand{\mbb}{\mathbb}

\newcommand{\fracddtz}{\frac{d}{dt}\bigg{|}_{t=0}}
\newcommand{\fracddtzs}{\frac{d}{dt}\big{|}_{t=0}}
\newcommand{\bigleftpar}{\big{(}}
\newcommand{\bigrightpar}{\big{)}}
\newcommand{\biggleftpar}{\bigg{(}}
\newcommand{\biggrightpar}{\bigg{)}}

\newcommand{\ttilde}{\widetilde}

\newcommand{\Wedge}{\bigwedge}

\makeatletter
\newcommand{\sectionnotoc}[1]{%
	\section*{#1}%
	\addtocontents{toc}{\protect\begingroup\protect\let\protect\contentsline\protect\@gobblethree\protect\endgroup}%
}
\makeatother

\makeatletter
\newcommand{\subsectionnotoc}[1]{%
	\subsection*{#1}%
	\addtocontents{toc}{%
		\protect\begingroup
		\protect\let\protect\contentsline\protect\@gobblethree
		\protect\endgroup
	}%
}
\makeatother

\DeclareMathOperator{\Ad}{Ad}
\DeclareMathOperator{\ad}{ad}

\DeclareMathOperator{\GL}{GL}
\DeclareMathOperator{\Gr}{Gr}
\DeclareMathOperator{\Hom}{Hom}

\DeclareMathOperator{\Aut}{Aut}
\DeclareMathOperator{\id}{id}

\DeclareMathOperator{\Id}{Id}

\DeclareMathOperator{\can}{can}

\DeclareMathOperator{\res}{res}

\DeclareMathOperator{\Inn}{Inn}

\theoremstyle{plain}
\newtheorem{thm}{Theorem}[section]

\theoremstyle{definition}
\newtheorem{defin}[thm]{Definition}

\newtheorem{lem}[thm]{Lemma}
\newtheorem{prop}[thm]{Proposition}

\newtheorem{rem}[thm]{Remark}

\title{On the Moser Trick for Lie Subalgebras and Foliations}
\author{Ilias Ermeidis}
\address{Department of Mathematics, Aristotle University of Thessaloniki, Thessaloniki 54124, Greece}
\email{ermeidis95@gmail.com}
\date{}
\thanks{Part of this work was carried out during the author’s PhD studies at the University of G\"ottingen, partially funded by a GSSP-DAAD fellowship and by RTG 2491 (G\"ottingen).}

\begin{document}

\begin{abstract}
Given a smooth deformation of a Lie subalgebra, we establish a necessary and sufficient condition for its smooth triviality and derive an analogous criterion for Lie ideals. We then give a direct proof of the Moser trick for foliations, which forms the basis for extending this result to general Lie subalgebroids.
\end{abstract}

\maketitle

%\medskip

%\tableofcontents 

\section{Introduction}

Lie algebroids constitute a central object of study in differential geometry and mathematical physics, as they encode infinitesimal symmetries across a broad class of geometric and physical settings. The two limiting cases of Lie algebroids are, on the one hand, Lie algebras and, on the other, tangent bundles. Accordingly, the corresponding extreme cases of Lie subalgebroids are Lie subalgebras and foliations, respectively. These corner cases are precisely the main ingredients of the present note. 

In the case of Lie subalgebras, the deformation cohomology was investigated by Richardson \cite{Richardson-Deformations-of-subalgebras-69'} while the $L_\infty$-algebra structure on the corresponding cochain complex was later established by Frégier–Zambon in \cite{Zambon-Fregier-SimultaneousDefOfAlgebraAndMorphisms-15'}. Furthermore, (topological) rigidity and stability results in the context of Lie subalgebras originally appeared in \cite{Richardson-A-Rigidity-theorem-for-subalgebras-of-Lie-and-associative-algebras-67', Richardson-Deformations-of-subalgebras-69'}, were reformulated, along with many other related and new results, in a modern differential-geometric language by Crainic–Schätz–Struchiner in \cite{Crainic-Ivan-Schaetz}.

In Section \ref{second section of the paper deformations of foliations}, we associate a family of cocycles to a smooth deformation of a Lie subalgebra and use this infinitesimal data to obtain a cohomological criterion for its triviality (Theorem \ref{Moser trick for Lie subalgebras}). We also prove an analogous result for Lie ideals (Theorem \ref{Moser trick for Lie ideals}). In contrast to the space of Lie algebra structures on a vector space, it is unknown whether the spaces of Lie subalgebras and Lie ideals of a Lie algebra are locally path-connected. Therefore, the smooth deformation-triviality results established in this paper cannot be recovered from the more topologically oriented approaches of \cite{Richardson-A-Rigidity-theorem-for-subalgebras-of-Lie-and-associative-algebras-67', Richardson-Deformations-of-subalgebras-69', Crainic-Ivan-Schaetz,Ermeidis-Jotz-Deformations-of-ideals-in-Lie-algebras-PREPRINT-2024'}. However, our techniques are substantially inspired by those of \cite{CardenasIvanStabilityofLiegrouphomomorphisms-2020}, where Cárdenas-Struchiner proved a Moser trick for Lie group homomorphisms and Lie subgroups. Namely, Proposition \ref{Moser lemma for Lie algebra morphisms} may be regarded as an infinitesimal analogue of \cite[Theorem 3.7]{CardenasIvanStabilityofLiegrouphomomorphisms-2020}.

Moving on to the case of foliations, their infinitesimal deformation theory was developed by Heitsch in \cite{Heitsch-A-cohomology-for-foliated-manifolds-1975}. The richer algebraic structure carried by the deformation complex of a foliation has been investigated by several authors using different, a priori independent techniques, including \cite{Huebschmann-Higher-homotopies-and-MC-algebras-the-foliation-paper-05'} by Huebschmann, \cite{Vitagliano-On-the-strong-homotopy-Lie-Rinehart-algebra-of-foliation-14'} by Vitagliano, and \cite{Xiang-Ji-Simultaneous-deformations-of-Lie-algebroid-and-subalgebroid-14'} by Ji, where the latter work presents the structure in the broader context of Lie subalgebroids.

Crainic-Moerdijk solved the deformation problem for Lie algebroids \cite{Crainic-Moerdijk-Deformations-08'} and developed the corresponding adapted Moser trick \cite[Theorem 2]{Crainic-Moerdijk-Deformations-08'}. Furthermore, the authors proved that there exists a quasi-isomorphism between the deformation complex of a foliation, viewed as a Lie subalgebroid of the tangent bundle, and the deformation complex of the foliation viewed as a Lie algebroid in its own right; this quasi-isomorphism also maps deformation cocycles to deformation cocycles \cite[Proposition 4]{Crainic-Moerdijk-Deformations-08'}. Together, these results imply that the Moser trick for foliations holds in the expected and most natural way.

The main aim of Section \ref{section 3 of the paper on deformations of foliations} is to provide a self-contained proof of the Moser trick for foliations (Theorem \ref{Moser Theorem for foliations}), expressed solely in terms of the deformation complex of the foliation, regarded as an involutive distribution. This approach furnishes a convenient framework for extending the result to arbitrary Lie subalgebroids \cite{Ermeidis-Batakidis}, where the aforementioned quasi-isomorphism generally fails to hold.

 \subsection*{Acknowledgement} Section \ref{section 3 of the paper on deformations of foliations} is largely based on Chapter 2 of the author's PhD thesis \cite{Ermeidis-PhD-thesis-2025}, completed under the supervision of Madeleine Jotz. The author is grateful to her for helpful comments on an earlier draft of this section and for a stimulating collaboration on deformations of Lie ideals, which led to \cite{Ermeidis-Jotz-Deformations-of-ideals-in-Lie-algebras-PREPRINT-2024'}. This work, in turn, inspired the author to prove Theorem \ref{Moser trick for Lie ideals}.

\section{Triviality of smooth deformations of Lie subalgebras and Lie ideals}\label{second section of the paper deformations of foliations}
%The $k$-Grassmannian of a Lie algebra $\got{g}$ is denoted by $\Gr_k(\got{g})$. Recall that the $k$-Grassmannian of an $n$-dimensional vector space $\got{g}$ is the connected, compact smooth manifold of dimension $k\cdot(n-k)$ whose points are $k$-dimensional vector subspaces of $\got{g}$, see e.g., \cite{Lee-Smooth-manifolds-2013}.
The Lie bracket of a finite dimensional real Lie algebra $\got{g}$ is denoted by $\mu_{\got{g}}:\Wedge^2\got{g}\to\got{g}$, and its adjoint representation by $\ad:\got{g}\to\got{gl}(\got{g})$. Furthermore, \(G\) denotes the unique simply connected Lie group integrating \(\mathfrak{g}\), and $\Ad:G\to\GL(\got{g})$ its adjoint representation. The Chevalley-Eilenberg complex of $\got{g}$ with values in a representation $r:\got{g}\to\got{gl}(V)$ is denoted by $C^{\bullet}(\got{g};V)\eqdef\Wedge^\bullet\got{g}^*\otimes V$, its differential by $\delta_{\got{g}}^V:C^\bullet(\got{g}; V)\to C^{\bullet+1}(\got{g};V)$, and its space of cocycles by $Z^\bullet(\got{g}; V)$. Given two Lie algebras $\mathfrak{h}$ and $\mathfrak{g}$, we denote by $\mathrm{Hom}(\mathfrak{h}, \mathfrak{g})$ the vector space of linear maps from $\mathfrak{h}$ to $\mathfrak{g}$. All smooth deformations appear in this section are parametrized by an open interval \(I \subset \mathbb{R}\) containing the origin. Throughout, a \emph{smooth family} $(s_t)_{t\in I}$ in a finite-dimensional vector space $S$ is a smooth curve $I \to S$, $t \mapsto s_t$, and will be denoted by $(s_t)_{t \in I} \subset S$.\\

The following list summarizes all the deformation-theoretic terminology required for the purposes of this section. 
\vspace*{+0.1cm}
\begin{enumerate}[label=(\Alph*).]
	\item A smooth deformation $(\got{g}_t)_{t\in I}\eqdef(\got{g},\mu_{\got{g}}^t)_{t\in I}$ of a Lie algebra $(\got{g},\mu_{\got{g}})$ is a smooth family $(\mu_{\got{g}}^t)_{t\in I}\subset\Hom(\wedge^2\got{g},\got{g})$ such that $\mu_{\got{g}}^0=\mu_{\got{g}}$ and, for each $t\in I$, $\mu_{\got{g}}^t:\wedge^2\got{g}\to\got{g}$ is a Lie bracket.
	The deformation is called trivial, if there exists a smooth family $(\psi_t)_{t\in I}\subset\GL(\got{g})$ such that $\psi_0=\id_{\got{g}}$ and $\psi_t:\got{g}\to(\got{g},\mu_{\got{g}}^t)$ is a Lie algebra isomorphism for all $t\in I$.\\ 
	For each $t\in I$, the differential of the Chevalley-Eilenberg complex $C^\bullet(\got{g}_t;\got{g}_t)$ with values in its adjoint representation, is denoted by $\delta_{\mu_{\got{g}}^t}$.
\vspace*{+0.2cm}
\item A smooth deformation $(\phi_t)_{t\in I}$ of a Lie algebra morphism $\phi:\got{h}\to\got{g}$ is a smooth family $(\phi_t)_{t\in I}\subset\Hom(\got{h},\got{g})$ such that $\phi_0=\phi$ and, for each $t\in I$, $\phi_t:\got{h}\to\got{g}$ is a Lie algebra morphism.\\
The deformation is called $\Inn(\got{g})$-trivial if there exists a smooth family $(g_t)_{t\in I}\subset G$ such that $g_0=1_G$ and $\phi_t=\Ad_{g_t}\circ\phi$ for all $t\in I$. More generally, it is called $\Aut(\got{g})$-trivial if there exists a smooth family $(\Psi_t)_{t\in I}\subset\Aut(\got{g})$ such that $\Psi_0=\id_{\got{g}}$ and $\phi_t=\Psi_t\circ\phi$ for all $t\in I$.\\
For each $t\in I$, we denote by $C^{\bullet}_{\phi_t}(\got{h};\got{g})$, the Chevalley-Eilenberg complex of $\got{h}$ with values in the representation $(\phi_t^*\circ\ad):\got{h}\to\got{gl}(\got{g})$, defined by
\begin{equation*}
(\phi_t^*\circ\ad)_u(v)\eqdef\ad_{\phi_t(u)}(v), \ \ \text{for all} \ u,v\in\got{h}.
\end{equation*}
Its differential is denoted by $\delta_{\phi_t}$.
\vspace*{+0.2cm}
\item Let $\iota:(\got{h},\mu_{\got{h}})\to\got{g}$ be an injective Lie algebra morphism. A smooth deformation $(\got{h}_t,\iota_t)_{t\in I}$ of the Lie subalgebra $(\got{h},\iota)$ consists of a smooth deformation $(\got{h}_t)_{t\in I}\eqdef(\got{h},\mu_{\got{h}}^t)_{t\in I}$ of the Lie algebra $(\got{h},\mu_{\got{h}})$, together with a smooth family $(\iota_t)_{t\in I}\subset\Hom(\got{h},\got{g})$ such that $\iota_0=\iota$ and $\iota_t:\got{h}_t\to\got{g}$ is an injective Lie algebra morphism for all $t\in I$.\footnote{For the equivalence between this definition of a Lie subalgebra and the one in terms of a curve in the Grassmannian, we refer the reader to Appendix~\ref{appendix on curves in Grassmannian}.}\\
The deformation is called $\Aut(\got{g})$-trivial if there exist smooth families
$(\psi_t)_{t\in I}\subset\GL(\got{h})$ and $(\Psi_t)_{t\in I}\subset\Aut(\got{g})$ such that
$\psi_0=\id_{\got{h}}$, $\Psi_0=\id_{\got{g}}$, and, for each $t\in I$, the map
$\psi_t:\got{h}\to(\got{h},\mu_{\got{h}}^t)$ is a Lie algebra isomorphism satisfying
$\iota_t\circ\psi_t=\Psi_t\circ\iota$.\\
For each $t\in I$, we denote by $C^{\bullet}_{\iota_t}\bigleftpar\got{h}_t;\got{g}/{\iota_t(\got{h}_t)}\bigrightpar$, the Chevalley-Eilenberg complex of $\got{h}_t$ with values in the representation $	(\iota_t^*\circ\ad)^{\got{h}_t}:\got{h}_{t}\to\got{gl}\bigleftpar\got{g}/{\iota_t(\got{h}_t)}\bigrightpar$, defined by
\begin{equation*}
(\iota_t^*\circ\ad)^{\got{h}_t}_{u}\bigleftpar\pi_{\can}^{t}(x)\bigrightpar\eqdef\pi_{\can}^t\circ\ad_{\iota_t(u)}(x), \ \text{for all} \ u\in\got{h} \ \text{and} \ x\in\got{g},
\end{equation*}
where $\pi_{\can}^t:\got{g}\to{\got{g}/{\iota_t(\got{h}_t)}}$ is the canonical projection. Its differential is denoted by $\overline{\delta}_{\iota_t}$.
\vspace*{+0.2cm}
\item 	Let
\[
0 \longrightarrow \mathfrak{h} \stackrel{\iota}{\longrightarrow}\mathfrak{g} \stackrel{\pi}{\longrightarrow} \mathfrak{k} \longrightarrow 0
\]
be a short exact sequence of Lie algebras. A smooth deformation $(\got{h}_t,\iota_t,\got{\pi}_t,\got{k}_t)_{t\in I}$ of the Lie ideal $(\got{h},\iota,\pi,\got{k})$ consists of:

\begin{itemize}
	\item a smooth deformation \((\mathfrak{h}_t, \iota_t)_{t \in I}\) of the Lie subalgebra \((\mathfrak{h}, \iota)\),
	\item a smooth deformation \(\mathfrak{k}_t := (\mathfrak{k}, \mu_{\mathfrak{k}}^t)\) of the Lie algebra \((\mathfrak{k}, \mu_{\mathfrak{k}})\), and
	\item a smooth family \((\pi_t)_{t\in I} \subset \mathrm{Hom}(\mathfrak{g}, \mathfrak{k})\),
\end{itemize}
such that \(\pi_0 = \pi\) and, for each \(t \in I\), the following hold:
\begin{enumerate}
	\item $\pi_t : \mathfrak{g} \to \mathfrak{k}_t \ \ \text{is a surjective Lie algebra morphism}$, and
	\item $\iota_t(\mathfrak{h}_t) = \ker(\pi_t)$.
\end{enumerate}
The deformation is called $\Aut(\got{g})$-trivial if there exists a smooth family
\[
(\alpha_t,\beta_t,\gamma_t)_{t\in I}
\subset
\GL(\got{h})\times\Aut(\got{g})\times\GL(\got{k}),
\]
such that $\alpha_t:\got{h}\to\got{h}_t$ and $\gamma_t:\got{k}\to\got{k}_t$ are Lie algebra isomorphisms for all
$t\in I$, and the diagram
\[
\begin{tikzcd}
	0 \arrow[r] &
	\got{h} \arrow[r, "\iota"] \arrow[d, "\alpha_t"] &
	\got{g} \arrow[r, "\pi"] \arrow[d, "\beta_t"] &
	\got{k} \arrow[r] \arrow[d, "\gamma_t"] &
	0 \\
	0 \arrow[r] &
	\got{h}_t \arrow[r, "\iota_t"] &
	\got{g} \arrow[r, "\pi_t"] &
	\got{k}_t \arrow[r] &
	0
\end{tikzcd}
\]
commutes for all $t\in I$, with
\[
\alpha_0=\id_{\got{h}},\qquad
\beta_0=\id_{\got{g}},\qquad
\gamma_0=\id_{\got{k}}.
\]
For each $t\in I$, we denote by $C_{\pi_t}^\bullet\bigleftpar\got{g};\Hom(\got{h}_t,\got{k}_t)\bigrightpar$, the Chevalley--Eilenberg complex of $\got{g}$ with values in the induced $\Hom$-representation arising from the representations
\begin{align*}
\hspace*{+1cm}&\ad^{\got{h}_t}:\got{g}\to\got{gl}(\ker(\pi_t)), \quad x\mto\ad^{\got{h}_t}_x, \ \ \text{where} \ \ \ad^{\got{h}_t}_x(u)\eqdef\mu_{\got{g}}(x,\iota_t(u)), \ \ \text{for all} \ x\in\got{g}, \ u\in\got{h}, \ \text{and}\\
&\ad^{\got{k}_t}:\got{g}\to\got{gl}(\got{k}_t), \qquad \ \ \ \ x\mto\ad^{\got{k}_t}_x, \ \ \ \text{where} \ \ \ad^{\got{k}_t}_x\bigleftpar\pi_t(y)\bigrightpar\eqdef\pi_t\circ\mu_{\got{g}}(x,y), \ \ \text{for all} \ x\in\got{g}, \ y\in\got{g}.
\end{align*}
Its differential is denoted by $\overline{\delta}_{\pi_t}$.
\end{enumerate}

\begin{prop}\label{Moser lemma for Lie algebra morphisms}
For any Lie algebra morphism $\phi:\got{h}\to\got{g}$ and any smooth deformation $(\phi_t)_{t\in I}$ of $\phi$, the following statements hold.
	\begin{enumerate}[label=(\roman*)]
\item For each $t\in I$, the cochain
\[
\dot{\phi}_t \eqdef \frac{d}{dt}\phi_t \in C_{\phi_t}^{1}(\got{h};\got{g})
\]
is a cocycle.\vspace*{+0.1cm}

\item If $(\phi_t)_{t\in I}$ is $\Inn(\got{g})$-trivial, then there exists a smooth family
$(x_t)_{t\in I}\subset\got{g}$ such that
\[
\dot{\phi}_t=\delta_{\phi_t}(x_t)
\]
for all $t\in I$. Conversely, the existence of such a family implies that $(\phi_t)_{t\in I}$ is
trivial after possibly shrinking $I$.\vspace*{+0.1cm}

\item If $(\phi_t)_{t\in I}$ is $\Aut(\got{g})$-trivial, then there exist smooth
families $(X_t)_{t\in I}\subset Z^1(\got{g};\got{g})$ and
$(x_t)_{t\in I}\subset\got{g}$ such that
\[
\dot{\phi}_t=\phi_t^*(X_t)+\delta_{\phi_t}(x_t)
\]
for all $t\in I$. Conversely, the existence of such families implies that $(\phi_t)_{t\in I}$ is
$\Aut(\got{g})$-trivial after possibly shrinking $I$.
	\end{enumerate}
\end{prop}
\begin{proof}
	(i) By assumption, the following equation holds for all $t\in I$, and $u,v\in\got{h}$.
	\begin{equation*}
		\phi_t\bigleftpar\mu_{\got{h}}(u,v)\bigrightpar=\mu_{\got{g}}\bigleftpar\phi_t(u),\phi_t(v)\bigrightpar.
	\end{equation*}
	Differentiating with respect to $t$ both sides of the equation, we obtain:
	\begin{align*}
		\dot{\phi}_t\bigleftpar\mu_{\got{h}}(u,v)\bigrightpar&=\mu_{\got{g}}\bigleftpar\dot{\phi}_t(u),\phi_t(v)\bigrightpar+\mu_{\got{g}}\bigleftpar\phi_t(u),\dot{\phi}_t(v)\bigrightpar\Leftrightarrow\\
		\dot{\phi}_t\bigleftpar\mu_{\got{h}}(u,v)\bigrightpar&=-(\phi_t^*\circ\ad)_{v}(\dot{\phi}(u))+(\phi_t^*\circ\ad)_{u}(\dot{\phi}_t(v)).
	\end{align*}
	The last equation is equivalent to the condition $\delta_{\phi_t}(\dot{\phi}_t)=0$.\\
	
	(ii) Assuming that $(\phi_t)_{t\in I}$ is an $\Inn(\got{g})$-trivial deformation of $\phi:\got{h}\to\got{g}$, means that
	\begin{equation*}
		\phi_t=\Ad_{g_t}\circ\phi, \quad \text{for all} \ t\in I,
	\end{equation*}
	where $(g_t)_{t\in I}$ is a smooth curve on the Lie group $G$ such that $g_0=1_G$. As before, by differentiating the above equation with respect to $t$, we conclude that:
	\begin{align*}
		\dot{\phi}_t=\ad_{x_t}\circ(\Ad_{g_t}\circ\phi), \quad \text{where} \ \ x_t\eqdef\frac{d}{dt}g_t.
	\end{align*}
	The above equation can be rewritten as$$\dot{\phi}_t=\ad_{x_t}\circ\phi_t=-(\phi_t^*\circ\ad)_{(\cdot)}(x_t),$$which is equivalent to the condition $\dot{\phi}_t=\delta_{\phi_t}(x_t)$.\\
	
	Conversely, we now assume that
	\begin{equation*}
		\dot{\phi}_t=\delta_{\phi_t}(x_t) \quad \text{for all} \ \ t\in I.
	\end{equation*}
Consider the time-dependent vector field\footnote{A brief overview of time-dependent vector fields sufficient for the purposes of this note can be found in Appendix~\ref{appendix}.} on the Lie group $G$
$$
X:I\times G\to TG, \quad (t,\cdot)\mapsto X_t\eqdef X(t,\cdot), \quad \text{defined by} \  X_t\eqdef\overrightarrow{x_t},\quad \text{for all } t\in I,
$$
where $\overrightarrow{x_t}$ is the unique right-invariant vector field on $G$ associated with $x_t\in\got{g}$, i.e. $\overrightarrow{x_t}(1_G)=x_t.$\\
By the local existence theorem for flows, there exists $t_0>0$ such that the flow $\Phi_X^t\eqdef\Phi^{t,0}_X$ of $X$ is defined at $1_G\in G$ for all $t\in I_0\eqdef(-t_0,t_0)\subset I$. The integral curve of $X$, from time $0$ to time $t$, starting at $1_G\in G$, is then defined by
$$
(g_t)_{t\in I_0}\eqdef\Phi_X^t(1_G).
$$
We claim that $\phi_t=\Ad_{g_t}\circ\phi$ for small enough values of $t$. To prove this, we first consider the time-dependent vector field on $\got{g}$,$$Z:I_0\times\got{g}\to\got{g}, \quad (t,\cdot)\mto Z_t\eqdef Z(t,\cdot), \quad \text{defined by} \ Z_t\eqdef-\ad_{(\cdot)}(x_t).$$On the one hand, we observe, due to our initial assumption i.e. $\dot{\phi}_t=\delta_{\phi_t}(x_t)=-\ad_{\phi_t(\cdot)}(x_t)$, that the smooth family $(\dot{\phi}_t)_{t\in I_0}\subset\Hom(\got{h},\got{g})$ is equal, at each time $t$, to the pullback of $Z_t$ by $\phi_t$. On the other hand, the smooth family of cocycles associated with the trivial deformation $(\Ad_{g_t}\circ \phi)_{t\in I}$ of $\phi$, as follows from the proof of the direct impication above, is defined at each time $t\in I_0$, by:
\begin{equation*}
	Y_t\eqdef\delta_{\Ad_{g_t}\circ\phi}(x_t).
\end{equation*}
That means, the time-dependent vector field$$Y:I_0\times\got{g}\to\got{g}, \quad (t,\cdot)\mto Y(t,\cdot)\eqdef Y_t,$$coincides, at each time $t$, with the pullback of $Z_t$ by $(\Ad_{g_t}\circ\phi)$. Therefore, for any $u\in\got{h}$, the curves $t\mto\phi_t(u)$ and $t\mto\Ad_{g_t}\bigleftpar\phi(u)\bigrightpar$ are both integral curves of the time-dependent vector field $Z$ passing at the time equals $0$ through the point $\phi(u)\in\got{g}$, and so they must coincide for all small values of $t$.\\

\emph{(iii)} Assuming that $(\phi_t)_{t\in I}$ is an  $\Aut(\got{g})$-trivial deformation of $\phi$, means that 
\begin{equation*}
	\phi_t=\Psi_t\circ\phi, \quad \text{for all} \ t\in I,
\end{equation*}
where $(\Psi_t)_{t\in I}\subset\Aut(\got{g})$ such that $\Psi_0=\id_{\got{h}}$. Let $(X_t)_{t\in I}\subset\Hom(\got{g};\got{g})$ be the time-dependent vector field defined for each $t\in I$ by
\[
X_t\circ\Psi_t=\frac{d}{dt}\Psi_t.
\]
Since $(\Psi_t)_{t\in I}\subset\Aut(\got{g})$, it follows that $(X_t)_{t\in I}\subset Z^1(\got{g};\got{g})$. Differentiating now the equation $\phi_t=\Psi_t\circ\phi$ with respect to $t$, we obtain
\[
\dot{\phi}_t=X_t\circ\phi_t=\phi_t^*(X_t),
\]
for all $t\in I$, which proves the direct implication.\\

Conversely, we now assume that
\begin{equation}\label{in the Aut(g)-moser for Lie alg morph}
\dot{\phi}_t=\phi_t^*(X_t)+\delta_{\phi_t}(x_t) \quad \text{for all} \ \ t\in I,
\end{equation}
where $(X_t)_{t\in I}\subset Z^1(\got{g};\got{g})$ and $(x_t)_{t\in I}\subset\got{g}$ are both smooth families. We consider the flow $\Phi_X^t\eqdef\Phi_X^{t,0}$ from time $0$ to time $t$ of the time-dependent vector field
\begin{equation*}
	X:I\to \Hom(\got{g};\got{g}), \quad t\mto X(t)\coloneqq X_t,
\end{equation*}
determined uniquely be the equations
\begin{equation*}
	\frac{d}{dt}\Phi^t_X=X_t\circ\Phi_X^t, \quad \text{and} \quad \Phi_X^t=\id_{\got{g}}.
\end{equation*}
Since $(X_t)_{t\in I}\subset Z^1(\got{g};\got{g})$, it follows that
$(\Phi_X^t)_{t\in I}\subset\Aut(\got{g})$. This can be checked directly by first
noting that the equation
\begin{equation}\label{in the Aut(g)-moser trick for Lie algebra morphisms}
	\Phi_X^t\bigl(\mu_{\got{g}}(x,y)\bigr)
	=
	\mu_{\got{g}}\bigl(\Phi_X^t(x),\Phi_X^t(y)\bigr),
\end{equation}
holds for $t=0$ and all $x,y\in\got{g}$. Furthermore, differentiating the above
equation with respect to $t$, one sees that both sides have the same derivative;
hence \eqref{in the Aut(g)-moser trick for Lie algebra morphisms} holds for all
$t\in I$ and all $x,y\in\got{g}$.\\
Next and final step is to show that the Lie algebra morphism $\phi'_t\eqdef(\Phi_X^t)^{-1}\circ\phi_t$ is an $\Inn(\got{g})$-trivial deformation of $\phi$, for sufficiently small $t\in I$. Define $\widehat{x_t}\eqdef(\Phi_X^t)^{-1}(x_t)$, and rewrite equation \eqref{in the Aut(g)-moser for Lie alg morph}, as follows
\begin{equation*}
	\dot{\phi}_t=\phi_t^*(X_t)+\delta_{\phi_t}\bigleftpar\Phi_X^t(\widehat{x_t})\bigrightpar,
\end{equation*}
which is equivalent to
\begin{equation*}
	\dot{\phi}_t=\phi_t^*(X_t)+\Phi_X^t\bigleftpar\delta_{\phi'_t}(\widehat{x_t})\bigrightpar.
\end{equation*}
Moreover, differentiating both sides of the equation $\phi'_t=(\Phi_X^t)^{-1}\circ\phi_t$ with respect to $t$, we obtain
\begin{align*}
	\dot{\phi}'_t&=\biggleftpar\frac{d}{dt}(\Phi_X^t)^{-1}\biggrightpar\circ\phi_t+(\Phi_X^t)^{-1}\circ\dot{\phi}_t\\
	&=-(\Phi_X^t)^{-1}\circ\biggleftpar\frac{d}{dt}\Phi_X^t\biggrightpar\circ(\Phi_X^t)^{-1}\circ\phi_t+(\Phi_X^t)^{-1}\circ\dot{\phi}_t\\
	&=-(\Phi_X^t)^{-1}\circ X_t\circ\phi_t+(\Phi_X^t)^{-1}\circ\dot{\phi}_t\\
	&=\cancel{-(\Phi_X^t)^{-1}\circ \phi_t^*X_t}+\cancel{(\Phi_X^t)^{-1}\circ\phi_t^*X_t}+(\Phi_X^t)^{-1}\circ\Phi_X^t\bigleftpar\delta_{\phi'_t}(\widehat{x_t})\bigrightpar\\
	&=\delta_{\phi'_t}(\widehat{x_t}).
\end{align*}
Therefore, by part (ii), there exists a smooth family $(g_t)_{t\in I}\subset G$ with $g_0=1_G$, defined for sufficiently small $t\in I$, such that $(\Phi_X^t)^{-1}\circ\phi_t=\Ad_{g_t}\circ\phi$. Hence, $\phi_t = (\Phi_X^t \circ \Ad_{g_t})\circ \phi$, which completes the proof.
\end{proof}

\begin{thm}\label{Moser trick for Lie subalgebras}
	For any injective Lie algebra morphism $\iota:(\got{h},\mu_{\got{h}})\hookrightarrow\got{g}$ and any smooth deformation $(\got{h}_t,\iota_t)_{t\in I}$ of the Lie subalgebra $(\got{h},\iota)$, the following statements hold.
		\begin{enumerate}[label=(\roman*)]
			\item For each $t\in I$, the cochain
			\begin{equation*}
				\pi_{\can}^t\circ\dot{\iota}_t\in C^{\bullet}_{\iota_t}\bigleftpar\got{h}_t;\got{g}/{\iota_t(\got{h}_t)}\bigrightpar
			\end{equation*}
			is a cocycle.\vspace*{+0.1cm}
			\item If $(\got{h}_t,\iota_t)_{t\in I}$ is $\Aut(\got{g})$-trivial, then there exists a smooth family $(X_t)_{t\in I}\subset Z^1(\got{g};\got{g})$ such that
			\begin{equation*}
				\pi_{\can}^t\circ\dot{\iota}_t=\pi_{\can}^t\circ X_t\circ\iota_t.
			\end{equation*}
			Conversely, the existence of such a family implies that $(\got{h}_t,\iota_t)_{t\in I}$ is $\Aut(\got{g})$-trivial after possibly shrinking $I$.
		\end{enumerate}
\end{thm}
\begin{proof}
	\emph{(i)} By assumption, the following equation holds for all $t\in I$, and $u,v\in\got{h}$.
	\begin{equation*}
		\iota_t\bigleftpar\mu_{\got{h}}^t(u,v)\bigrightpar=\mu_{\got{g}}\bigleftpar\iota_t(u),\iota_t(v)\bigrightpar.
	\end{equation*}
	Differentiating with respect to $t$ both sides of the equation, we obtain:
	\begin{equation}\label{useful equation for later use in moser trick for lie subalgebras}
		\dot{\iota}_t\bigleftpar\mu_{\got{h}}^t(u,v)\bigrightpar+\iota_t\circ\frac{d}{dt}\mu_{\got{h}}^t(u,v)=\mu_{\got{g}}\bigleftpar\dot{\iota}_t(u),\iota_t(v)\bigrightpar+\mu_{\got{g}}\bigleftpar\iota_t(u),\dot{\iota}_t(v)\bigrightpar.
	\end{equation}
	Upon applying the canonical projection $\pi_{\can}^t$, we conclude that:
	\begin{align*}
			\pi_{\can}^t\circ\dot{\iota}_t\bigleftpar\mu_{\got{h}}^t(u,v)\bigrightpar&=\pi_{\can}^t\circ\mu_{\got{g}}\bigleftpar\dot{\iota}_t(u),\iota_t(v)\bigrightpar+\pi_{\got{g}/\got{h}_t}\circ\mu_{\got{g}}\bigleftpar\iota_t(u),\dot{\iota}_t(v)\bigrightpar\Leftrightarrow\\
			\pi_{\can}^t\circ\dot{\iota}_t\bigleftpar\mu_{\got{h}}^t(u,v)\bigrightpar&=-(\iota_t^*\circ\ad)^{\got{h}_t}_v(\dot{\iota}(u))+(\iota_t^*\circ\ad)^{\got{h}_t}_u(\dot{\iota}(v)).
	\end{align*}
	The last equation is equivalent to the condition $\overline{\delta}_{\iota_t}(\pi_{\can}^t\circ\dot{\iota_t})=0$.\\
	
	\emph{(ii)} 	Assuming that $(\got{h}_t,\iota_t)_{t\in I}$ is an $\Aut(\got{g})$-trivial deformation of $(\got{h},\iota)$, means that 
	\begin{equation*}
		\iota_t\circ\psi_t=\Psi_t\circ\iota, \quad \text{for all} \ t\in I,
	\end{equation*}
	where $(\Psi_t)_{t\in I}\subset\Aut(\got{g})$ satisfies $\Psi_0=\id_{\got{g}}$, and $(\psi_t)_{t\in I}\subset\GL(\got{h})$ is a family of Lie algebra isomorphisms $\psi_t:\got{h}\to(\got{h},\mu_{\got{h}}^t)$ with $\psi_0=\id_{\got{h}}$. Let $(X_t)_{t\in I}\subset\Hom(\got{g};\got{g})$ be the time-dependent vector field defined for each $t\in I$ by
	\[
	X_t\circ\Psi_t=\frac{d}{dt}\Psi_t.
	\]
	Since $(\Psi_t)_{t\in I}\subset\Aut(\got{g})$, it follows that $(X_t)_{t\in I}\subset Z^1(\got{g};\got{g})$. Differentiating now the equation $\iota_t\circ\psi_t=\Psi_t\circ\iota$ with respect to $t$, we obtain
	\begin{align*}
		\dot{\iota}_t\circ \psi_t+\iota_t\circ\dot{\psi}_t&=\biggleftpar\frac{d}{dt}\Psi_t\biggrightpar\circ\iota\Rightarrow\\
		\pi_{\can}^t\circ\dot{\iota}_t\circ\psi_t+ \cancel{\pi_{\can}^t\circ\iota_t\circ\dot{\psi}_t}&=\pi_{\can}^t\circ\biggleftpar\frac{d}{dt}\Psi_t\biggrightpar\circ\iota\Rightarrow\\
		\pi_{\can}^t\circ\dot{\iota}_t\circ\psi_t&=\pi_{\can}^t\circ X_t\circ\Psi_t\circ\iota\Rightarrow\\
		\pi_{\can}^t\circ\dot{\iota}_t\circ\psi_t&=\pi_{\can}^t\circ X_t\circ\iota_t\circ\psi_t\Rightarrow\\
		\pi_{\can}^t\circ\dot{\iota}_t&=\pi_{\can}^t\circ X_t\circ\iota_t,
	\end{align*}
	for all $t\in I$, which proves the direct implication.\\
	
	Conversely, we now assume that there exists a smooth family $(X_t)_{t\in I}\subset Z^1(\got{g};\got{g})$ such that
	\begin{equation*}
		\pi_{\can}^t\circ\dot{\iota}_t=\pi_{\can}^t\circ X_t\circ\iota_t, \quad \text{for all} \ \ t\in I,
	\end{equation*}
	which means the following equation holds for all $t\in I$ and $u\in\got{h}$:
	\begin{equation*}
		\pi_{\can}^t\bigleftpar X_t\circ\iota_t(u)-\dot{\iota}_t(u)\bigrightpar=0.
	\end{equation*}
	Therefore, we conclude that there exists a smooth family $(Y_t)_{t\in I}\subset\Hom(\got{h},\got{h})$, such that
	\begin{equation}\label{ekran-1}
		X_t\circ\iota_t-\dot{\iota}_t=\iota_t\circ Y_t.
	\end{equation}
	Let $\Phi_Y^t\eqdef\Phi_Y^{t,0}\in\GL(\got{h})$ be the flow of the time-dependent vector field $(Y_t)_{t\in I}\subset\Hom(\got{h},\got{h})$ from time $0$ to time $t$, that is
	\begin{equation*}
		\frac{d}{dt}\Phi_Y^{t}=Y_t\circ\Phi^t_Y \quad \text{and} \quad \Phi_Y^0=\id_{\got{h}},
	\end{equation*}
	defined for all $t$ in the maximal interval of existence containing $0$, denoted by $I_0\subset I$. We claim that, for each $t\in I_0$, $\Phi_Y^t:\got{h}\to(\got{h},\mu_{\got{h}}^t)$ is a Lie algebra isomorphism, that is
	\begin{equation}\label{ekran1}
		\Phi_Y^t\bigleftpar\mu_{\got{h}}(u,v)\bigrightpar=\mu_{\got{h}}^t\bigleftpar\Phi_Y^t(u),\Phi_Y^t(v)\bigrightpar, \quad \text{for all} \ t\in I_0 \ \ \text{and} \ \ u,v\in\got{h}.
	\end{equation}
	The equation \eqref{ekran1} clearly holds for $t=0$. In what follows, we show that both sides of (\ref{ekran1}) are integral curves of $Y$ with the same initial point at $t=0$. We start by differentiating the right-hand side with respect to $t$.
	\begin{align}\label{ekran1/2}
		\frac{d}{dt}\biggleftpar\mu_{\got{h}}^t\bigleftpar\Phi_Y^t(u),\Phi_Y^t(v)\bigrightpar\biggrightpar&=\frac{d}{dt}\mu_{\got{h}}^t\bigleftpar\Phi_Y^t(u),\Phi_Y^t(v)\bigrightpar+\mu_{\got{h}}^t\biggleftpar\frac{d}{dt}\Phi_Y^t(u),\Phi_Y^t(v)\biggrightpar+\mu_{\got{h}}^t\biggleftpar\Phi_Y^t(u),\frac{d}{dt}\Phi_Y^t\biggrightpar\nonumber\\
		&=\frac{d}{dt}\mu_{\got{h}}^t\bigleftpar\Phi_Y^t(u),\Phi_Y^t(v)\bigrightpar+\mu_{\got{h}}^t\bigleftpar Y_t\circ\Phi_Y^t(u),\Phi_Y^t(v)\bigrightpar+\mu_{\got{h}}^t\bigleftpar\Phi_Y^t(u),Y_t\circ\Phi_Y^t\bigrightpar.
	\end{align}
	We apply now the differential $\delta_{\iota_t}$ to both sides of the equation $\iota_t\circ Y_t=\dot{\iota}_t-X_t\circ\iota_t$.
	\begin{equation}\label{ekran2}
		\delta_{\iota_t}(\iota_t\circ Y_t)=\delta_{\iota_t}(X_t\circ\iota_t-\dot{\iota}_t).
	\end{equation}
	It is immediate to verify that for the left-hand side of equation \eqref{ekran2} holds the following
	\begin{equation}\label{ekran3}
		\delta_{\iota_t}(\iota_t\circ Y_t)=\iota_t\circ\delta_{\mu_{\got{h}}^t}\circ Y_t.
	\end{equation}
	Furthermore, for the right-hand side of \eqref{ekran2}, we obtain for all $u,v\in\got{h}$ that:
	\begin{align}\label{ekran4}
		\delta_{\iota_t}(X_t\circ\iota_t-\dot{\iota}_t)(u,v)&=\mathcolor{blue}{\mu_{\got{g}}(\iota_t(u),X_t\circ\iota_t(v))}\mathcolor{red}{-\mu_{\got{g}}(\iota_t(u),\dot{\iota}_t(v))}\nonumber\\
		& \ \ \ \mathcolor{blue}{+\mu_{\got{g}}(X_t\circ\iota_t(u),\iota_t(v))}\mathcolor{red}{-\mu_{\got{g}}(\dot{\iota}_t(u),\iota_t(v))}\nonumber\\
		& \ \ \ \mathcolor{blue}{-X_t\bigleftpar\mu_{\got{g}}(\iota_t(u),\iota_t(v))\bigrightpar}\mathcolor{red}{+\dot{\iota}_t(\mu_{\got{h}}^t(u,v))}\nonumber\\
		&=\mathcolor{blue}{0}\mathcolor{red}{-\delta_{\iota_t}(\dot{\iota}_t)(u,v)}.
	\end{align}
	But, equation (\ref{useful equation for later use in moser trick for lie subalgebras}) implies that
	\begin{equation}\label{ekran5}
		\delta_{\iota_t}(\dot{\iota}_t)=\iota_t\circ\frac{d}{dt}\mu_{\got{h}}^t.
	\end{equation}
	Due to equations \eqref{ekran3}, \eqref{ekran4}, and \eqref{ekran5}, equation \eqref{ekran2} is equivalent to
	\begin{equation*}
		\iota_t\circ\delta_{\mu_{\got{h}}^t}\circ Y_t=-\iota_t\circ\frac{d}{dt}\mu_{\got{h}}^t\Leftrightarrow\iota_t\biggleftpar\delta_{\mu_{\got{h}}^t}\circ Y_t+\frac{d}{dt}\mu_{\got{h}}^t\biggrightpar=0\Leftrightarrow\frac{d}{dt}\mu_{\got{h}}^t=-\delta_{\mu_{\got{h}}^t}\circ Y_t.
	\end{equation*}
	Going back to equation \eqref{ekran1/2}, and making use of our last equality, we have that
	\begin{align*}
		\frac{d}{dt}\biggleftpar\mu_{\got{h}}^t\bigleftpar\Phi_Y^t(u),\Phi_Y^t(v)\bigrightpar\biggrightpar&=-\delta_{\mu_{\got{h}}^t}(Y_t)\bigleftpar\Phi_Y^t(u),\Phi_Y^t(v)\bigrightpar+\mu_{\got{h}}^t\bigleftpar Y_t\circ\Phi_Y^t(u),\Phi_Y^t(v)\bigrightpar+\mu_{\got{h}}^t\bigleftpar\Phi_Y^t(u),Y_t\circ\Phi_Y^t(v)\bigrightpar,
	\end{align*}
	but\begin{equation*}
		-\delta_{\mu_{\got{h}}^t}(Y_t)\bigleftpar\Phi_Y^t(u),\Phi_Y^t(v)\bigrightpar=-\mu_{\got{h}}^t\bigleftpar\Phi_Y^t(u),Y_t\circ\Phi_Y^t(v)\bigrightpar-\mu_{\got{h}}^t\bigleftpar Y_t\circ\Phi_Y^t(u),\Phi_Y^t(v)\bigrightpar+Y_t\bigleftpar\mu_{\got{h}}^t\bigleftpar\Phi_Y^t(u),\Phi_Y^t(v)\bigrightpar\bigrightpar,
	\end{equation*}
	and so we conclude that
	\begin{equation*}
		\frac{d}{dt}\biggleftpar\mu_{\got{h}}^t\bigleftpar\Phi_Y^t(u),\Phi_Y^t(v)\bigrightpar\biggrightpar=Y_t\bigleftpar\mu_{\got{h}}^t\bigleftpar\Phi_Y^t(u),\Phi_Y^t(v)\bigrightpar\bigrightpar,\quad \text{for all} \ t\in I_0 \ \ \text{and} \ \ u,v\in\got{h}.
	\end{equation*}
	Taking now the derivative with respect to $t$ of the left-hand side of \eqref{ekran1}, we get
	\begin{equation*}
		\frac{d}{dt}\Phi_Y^t\bigleftpar\mu_{\got{h}}(u,v)\bigrightpar=Y_t\bigleftpar\Phi_Y^t\bigleftpar\mu_{\got{h}}(u,v)\bigrightpar\bigrightpar, \quad \text{for all} \ t\in I_0 \ \ \text{and} \ \ u,v\in\got{h}.
	\end{equation*}
	Therefore, we have shown that, for any $u,v\in\got{h}$, the curves $t\mto\Phi_Y^t\bigl(\mu_{\mathfrak{h}}(u,v)\bigr)$ and $t\mto\mu_{\mathfrak{h}}^t\bigl(\Phi_Y^t(u),\Phi_Y^t(v)\bigr)$ are both integral curves of the time-dependent vector field $(Y_t)_{t\in I}\subset \mathrm{Hom}(\mathfrak{h},\mathfrak{h})$, starting from the same point at $t=0$. By uniqueness of solutions, they coincide for all $t\in I_0$.\\
	We now obtain the desired conclusion from the following short computation, which holds for all $t\in I$.
	\begin{align*}
		\frac{d}{dt}(\iota_t\circ\Phi_Y^t)&=\dot{\iota}_t\circ\Phi_Y^t+\iota_t\circ\frac{d}{dt}\Phi_{{Y}}^t\\
		&\stackrel{\eqref{ekran-1}}{=}X_t\circ\iota_t\circ\Phi_t^Y-\iota_t\circ Y_t\circ\Phi_Y^t+\iota_t\circ\frac{d}{dt}\Phi_Y^t\\
		&=X_t\circ(\iota_t\circ\Phi_t^Y)-\cancel{\iota_t\circ Y_t\circ\Phi_Y^t}+\cancel{\iota_t\circ Y_t\circ\Phi_Y^t}\\
		&=(\iota_t\circ\Phi_t^Y)^*(X_t).
	\end{align*}
	By part (iii) of Proposition \ref{Moser lemma for Lie algebra morphisms}, it follows that for sufficiently small values of \(t\), the Lie algebra morphism \(\iota_t\circ\Phi_Y^t:\got{h}\to\got{g}\) is $\Aut(\got{g})$-trivial, that is, \(\iota_t\circ\Phi_Y^t=\Psi_t\circ\iota\) (see the proof of part (iii) in Proposition \ref{Moser lemma for Lie algebra morphisms} for how $(\Psi_t)_{t\in I}\subset\Aut(\got{g})$ is obtained), which implies that the deformation \((\got{h}_t,\iota_t)_{t\in I}\) of the Lie subalgebra $(\got{h},\iota)$ is $\Aut(\got{g})$-trivial for sufficiently small values of \(t\).
\end{proof}

\begin{thm}\label{Moser trick for Lie ideals}
	For any short exact sequence of Lie algebras
	\[
	0 \longrightarrow \mathfrak{h} \stackrel{\iota}{\longrightarrow}\mathfrak{g} \stackrel{\pi}{\longrightarrow} \mathfrak{k} \longrightarrow 0,
	\]
	and any smooth deformation $(\got{h},\iota_t,\pi_t,\got{k}_t)_{t\in I}$ of the Lie ideal $(\got{h},\iota,\pi,\got{k})$, the following statements hold.
			\begin{enumerate}[label=(\roman*)]
		\item For each $t\in I$, the cochain
		\begin{equation*}
			\dot{\pi}_t\circ\iota_t\in C_{\pi_t}^\bullet\bigleftpar\got{g};\Hom(\got{h}_t,\got{k}_t)\bigrightpar
		\end{equation*}
		is a cocycle.\vspace*{+0.1cm}
		\item If $(\got{h}_t,\iota_t,\pi_t,\got{k}_t)_{t\in I}$ is $\Aut(\got{g})$-trivial, then there exists a smooth family $(X_t)_{t\in I}\subset Z^1(\got{g};\got{g})$ such that
		\begin{equation*}
			\dot{\pi}_t\circ\iota_t=-\pi_t\circ X_t\circ\iota_t.
		\end{equation*}
		Conversely, the existence of such a family implies that $(\got{h}_t,\iota_t,\pi_t,\got{k}_t)_{t\in I}$  is $\Aut(\got{g})$-trivial after possibly shrinking $I$.
	\end{enumerate}
\end{thm}
\begin{proof}
		\emph{(i)} By assumption, the following equation holds for all $t\in I$, and $x,y\in\got{g}$.
	\begin{equation*}
		\pi_t\bigleftpar\mu_{\got{g}}(x,y)\bigrightpar=\mu^t_{\got{k}}\bigleftpar\pi_t(x),\pi_t(y)\bigrightpar
	\end{equation*}
Differentiating with respect to $t$ both sides of the equation, we obtain:
	\begin{equation}
\dot{\pi}_t\bigleftpar\mu_{\got{g}}(x,y)\bigrightpar=\biggleftpar\frac{d}{dt}\mu_{\got{k}}^t\biggrightpar\bigleftpar\pi_t(x),\pi_t(y)\bigrightpar+\mu^t_{\got{k}}\bigleftpar\dot{\pi}_t(x),\pi_t(y)\bigrightpar+\mu^t_{\got{k}}\bigleftpar\pi_t(x),\dot{\pi}_t(y)\bigrightpar.
	\end{equation}
Assume that \(y\in\ker(\pi_t)\), that is, \(y=\iota_t(u)\) for some unique \(u\in\mathfrak{h}\). Then the above equation simplifies to
	\begin{align*}
	\dot{\pi}_t\bigleftpar\mu_{\got{g}}(x,\iota_t(u))\bigrightpar=\mu^t_{\got{k}}\bigleftpar\pi_t(x),\dot{\pi}_t\circ\iota_t(u)\bigrightpar.
	\end{align*}
	The last equation is equivalent to the condition $\overline{\delta}_{\pi_t}(\dot{\pi}_t\circ\iota_t)=0$.\\
	
		\emph{(ii)} Assume that $(\got{h}_t,\iota_t,\pi_t,\got{k}_t)_{t\in I}$ is an $\Aut(\got{g})$-trivial deformation of $(\got{h},\iota,\pi,\got{k})$. Then there exists a smooth family $(\alpha_t,\beta_t,\gamma_t)_{t\in I}\subset \GL(\got{h})\times\Aut(\got{g})\times\GL(\got{k})$ such that $\alpha_0=\id_{\got{h}}$, $\beta_0=\id_{\got{g}}$, and $\gamma_0=\id_{\got{k}}$, satisfying, for all $t\in I$:
		\begin{align*}
		\ \beta_t\circ\iota=\iota_t\circ\alpha_t, \quad \text{and} \quad \ \gamma_t\circ\pi=\pi_t\circ\beta_t
		\end{align*}
	Differentiating the identity $\gamma_t\circ\pi=\pi_t\circ\beta_t$ with respect to $t$, we obtain
		\begin{align*}
			\dot{\gamma}_t\circ\pi&=\dot{\pi}_t\circ\beta_t+\pi_t\circ\dot{\beta}_t\Rightarrow\\
				\dot{\gamma}_t\circ\pi\circ\beta_t^{-1}&=\dot{\pi}_t+\pi_t\circ\dot{\beta}_t\circ\beta_t^{-1}\Rightarrow\\
				\dot{\gamma}_t\circ\pi\circ\beta_t^{-1}\circ\iota_t&=\dot{\pi}_t\circ\iota_t+\pi_t\circ\underbrace{\dot{\beta}_t\circ\beta_t^{-1}}_{\eqdef X_t}\circ\iota_t\Rightarrow\\
					\cancel{\dot{\gamma}_t\circ\pi\circ\iota\circ\alpha_t^{-1}}&=\dot{\pi}_t\circ\iota_t+\dot{\pi}_t\circ X_t\circ\iota_t\Rightarrow\\
					\dot{\pi}_t\circ\iota_t&=-\dot{\pi}_t\circ X_t\circ\iota_t, \quad \text{for all} \ \ t\in I,
		\end{align*}
and this concludes the direct implication. Observe that $(X_t)_{t\in I}\subset Z^1(\got{g};\got{g})$, since $\dot{\beta}_t$ is a $1$-cocycle and $X_t$ is its pullback under the Lie algebra automorphism $\beta_t^{-1}\in\Aut(\got{g})$, for every $t\in I$.\\

Conversely, we now assume that there exists a smooth family $(X_t)_{t\in I}\subset Z^1(\got{g};\got{g})$ such that
\begin{equation*}
		\dot{\pi}_t\circ\iota_t=-\pi_t\circ X_t\circ\iota_t, \quad \text{for all} \ \ t\in I.
\end{equation*}
Using the identity $\dot{\pi}_t\circ\iota_t=-\pi_t\circ\dot{\iota}_t$ in the equation above, we obtain
\begin{equation*}
	\pi_t\circ\dot{\iota}_t=\pi_t\circ X_t\circ\iota_t,\quad \text{for all } t\in I,
\end{equation*}
which is equivalent saying that the following equation holds for all $t\in I$ and $u\in\got{h}$:
\begin{equation*}
	\pi_t\bigleftpar X_t\circ\iota_t(u)-\dot{\iota}_t(u)\bigrightpar=0.
\end{equation*}
Therefore, we conclude that there exists a smooth family $(Y_t)_{t\in I}\subset\Hom(\got{h},\got{h})$, such that
\begin{equation}
	X_t\circ\iota_t-\dot{\iota}_t=\iota_t\circ Y_t.
\end{equation}
By the proof of Theorem \ref{Moser trick for Lie subalgebras}, specifically from equation \eqref{ekran-1} onwards, we know that there exists an open interval $I'\subset I_0$ containing $0$ and a smooth family $(\Psi_t)_{t\in I'}\subset\Aut(\got{g})$ with $\Psi_0=\id_{\got{g}}$ such that
\begin{equation*}
	\iota_t\circ\Phi_Y^t=\Psi_t\circ\iota,
\end{equation*}
where $\Phi_Y^t\eqdef\Phi_Y^{t,0}\in\GL(\got{h})$ denotes the flow of the time-dependent vector field $(Y_t)_{t\in I}\subset\Hom(\got{h},\got{h})$ from time $0$ to time $t$, and $I_0$ is the maximal open interval on which this flow is defined. Finally, we define a smooth family $(\overline{\Psi_t})_{t\in I'}\subset\GL(\got{k})$ by the formula
\begin{equation*}
	\overline{\Psi_t}\circ\pi\eqdef\pi_t\circ\Psi_t.
\end{equation*}
It is immediate to verify that $\overline{\Psi_t}$ is well-defined for all $t\in I'$, that $\overline{\Psi_0}=\id_{\got{k}}$, and that, for each $t\in I'$, $\Psi_t:\got{k}\to(\got{k},\mu_{\got{k}}^t)$ is a Lie algebra isomorphism, since it is a composition of Lie algebra isomorphisms. The proof is now complete, since the smooth family
\begin{equation*}
	(\alpha_t,\beta_t,\gamma_t)_{t\in I'}\eqdef(\Phi_Y^t,\Psi_t,\overline{\Psi_t})_{t\in I'}\subset\GL(\got{h})\times\Aut(\got{g})\times\GL(\got{k}),
\end{equation*}
establishes the $\Aut(\got{g})$-triviality of the smooth deformation $(\got{h}_t,\iota_t,\pi_t,\got{k}_t)_{t\in I}$ of the Lie ideal $(\got{h},\iota,\pi,\got{k})$ for sufficiently small $t$.
\end{proof}

\section{Moser trick for foliations}\label{section 3 of the paper on deformations of foliations}

Let $\mathcal{F}$ be a $k$-dimensional foliation on an $n$-dimensional manifold $M$, and denote by $F \coloneqq T\mathcal{F}$ its associated involutive distribution of $TM$. In what follows, we will occasionally refer to the involutive subbundle associated with a foliation also as a foliation; the distinction will always be clear from the aforementioned notation. As in the previous section, smooth deformations of foliations are parametrized by an open interval $I\subset\mbb{R}$ containing the origin.

\begin{defin}[\cite{delHoyo-Fernades-deformations-of-compact-foliations-2019}]\label{definition of a deformation of a foliation, global version}
	A \emph{smooth deformation of a foliation $\cali{F}$ on $M$} is a foliation $\widetilde{\cali{F}}$ on $M\times I$ such that  ${\cali{F}}_t\eqdef\ttilde{\cali{F}}|_{M\times\{t\}}$ is tangent to the slices $M\times\{t\}$ for all $t\in I$ and ${\cali{F}}_0={\cali{F}}$.
\end{defin}

\begin{rem}\label{equivalence of deforming the foliation as a foliation and as an involutive distribution}
	Since $\ttilde{\cali{F}}$ is tangent to the slices, that means if $\{\widetilde{e_i}\}_{i=1}^k$ is a local frame of $\ttilde{F}\eqdef T\widetilde{\cali{F}}$ around $U\times J\subset M\times I$. Hence, it should have the following form:
	\begin{equation}\label{formula for a local frame of the total foliation associated to a deformation of a foliation}
		\ttilde{e_i}:U\times J\to TM\times TI, \quad \ttilde{e_i}(p,t)=(e_i(p,t), 0_t),
	\end{equation}
	where $e_i:U\times J\to TM$ is a time-dependent vector field such that $\{e_i(\cdot, t)\eqqcolon e^t_i\}_{i=1}^k$ is a local frame for the vector subbundle $F_t\eqdef T\cali{F}_t\subset TM\times TI$ over $M\times\{t\}\simeq M$. Now, involutivity of $\ttilde{F}$ imposes that, for any $i,j\in\{1,\dots, k\}$, $[\ttilde{e_i},\ttilde{e_j}]\in\Gamma(\ttilde{F})$ where the bracket is the Lie bracket of vector fields on $M\times I$. Due to \eqref{formula for a local frame of the total foliation associated to a deformation of a foliation}, it is immediate that the involutivity condition of $\ttilde{F}$ is equivalent to the involutivity condition for each $F_t$ i.e. $[e_i^t,e_j^t]\in\Gamma(F_t)$ for each $t\in I$. Therefore, $(F_t)_{t\in I}$ is a smooth family of involutive subbundles in the sense that there exists a time-dependent local frame that is a local frame of $F_t$ for each $t\in I$. This definition for deformations of foliations, equivalent to Definition~\ref{definition of a deformation of a foliation, global version}, was first introduced by Heitsch in \cite{Heitsch-A-cohomology-for-foliated-manifolds-1975}. Sometimes we refer to $\ttilde{F}\eqdef T\widetilde{\cali{F}}$ as the total foliation associated to $(F_t)_{t\in I}\eqdef(T\cali{F}_t)_{t\in I}$.
\end{rem}

Given an involutive subbundle $F\eqdef T\cali{F}$ of $TM$, denote the normal bundle of the foliation $\mathcal{F}$ by $\nu(\mathcal{F}) \coloneqq TM / F$, and let $v \mapsto \overline{v}$ be the canonical projection $TM \to \nu(\mathcal{F})$, for any $v \in TM$. There is a natural representation of $F\eqdef T\cali{F}$ on the normal bundle $\nu(\mathcal{F})$: the Bott connection $\nabla^{F} \colon \Gamma(F) \times \Gamma(\nu(\mathcal{F})) \to \Gamma(\nu(\mathcal{F}))$, which is originally introduced in \cite{Bott1972characteristicclasses-1972} as follows:
\begin{equation*}
\nabla^F_{X}\overline{Y}=\overline{[X,Y]}, \quad \text{for any} \ X\in\Gamma(F) \ \text{and} \ Y\in\mathfrak{X}(M).
\end{equation*}
In \cite{Heitsch-A-cohomology-for-foliated-manifolds-1975}, it is proven that the deformation complex of foliations is $(\Omega^\bullet(\cali{F};\nu(\cali{F})),d_{\nabla^F})$, in the sense of the following Lemma \ref{cocycle proposition for foliations}. Its corresponding cohomology is denoted analogously by $H^{\bullet}(\cali{F};\nu(\cali{F}))$. We pick a non-degenerate metric on $M$ and we consider the orthogonal complement $F^\perp$ with respect to $F$. From now on, we use implicitly the canonical identification $F^\perp\simeq \nu(\cali{F})$. We have two families of orthogonal projections, with respect to the metric, denoted by: $$\text{(i)} \ \pi_t:TM\to F_t,\quad \text{and} \quad \text{(ii)} \ \pi_t^\perp:TM\to F_t^\perp\simeq TM/F_t\eqqcolon\nu(\cali{F}_t).$$

\begin{lem}[\cite{Heitsch-A-cohomology-for-foliated-manifolds-1975}]\label{lemma for differentiating the orthogonal projections in deformations of foliations}
(1) If $X\in\Gamma(F_t)$, then $\frac{d}{dt}\pi_t(X)\in\Gamma(\nu(\cali{F}_t))$.\\
\hspace*{+2.9cm} (2)$\frac{d}{dt}\pi_t^\perp(X)=-\frac{d}{dt}\pi_t(X)$, for any $X\in\got{X}(M)$.
\end{lem}
%\begin{proof}
%	(1) Differentiating with respect to $t$ the equation $\pi_t\circ\pi_t=\pi_t$, we get $$\bigg{(}\frac{d}{dt}\pi_t\bigg{)}\circ\pi_t+\pi_t\circ\bigg{(}\frac{d}{dt}\pi_t\bigg{)}=\frac{d}{dt}\pi_t.$$ Since $\pi_t(X)=X$ for any $X\in\Gamma(F_t)$, we obtain that $\pi_t\circ\tfrac{d}{dt}\pi_t(X)=0$ or equivalently $\frac{d}{dt}\pi_t(X)\in\nu(\cali{F}_t)$.\\ \\
%	(2) Differentiating with respect to $t$ the equation $\pi_t+\pi_t^\perp=\id$, gives us that $\frac{d}{dt}\pi_t^\perp=-\frac{d}{dt}\pi_t$.
%\end{proof}

\begin{prop}[\cite{Heitsch-A-cohomology-for-foliated-manifolds-1975}]\label{cocycle proposition for foliations}
	If $(F_t)_{t\in I}$ is a smooth deformation of $F$, then
	\begin{equation*}
		\sigma_0(X)=\frac{d}{dt}\bigg{|}_{t=0}\pi_t(X), \quad X\in\Gamma(F),
	\end{equation*}
	is a $1$-cocycle in $\Omega^{\bullet}(\cali{F};\nu(\cali{F}))$.
\end{prop}
%\begin{proof}
%	Since for each $t\in I$, the subbundle $F_t$ is involutive, the map $\Lambda_t\in\Omega^2(M; \nu(\cali{F}_t))$ given by
%	\begin{equation*}
%		\Lambda_t(X,Y)=\pi_t^\perp[\pi_t(X),\pi_t(Y)]
%	\end{equation*}
%	vanishes identically for all $t\in I$ and $X,Y\in\got{X}(M)$. Assuming that $X,Y\in\Gamma(F)$ and differentiating the above expression at $t=0$, we obtain, using that $\pi_0(X)=X$ and $\pi_0(Y)=Y$, that
%	\begin{align*}
%		\fracddtz\pi_t([X,Y])=\pi_0^\perp\bigg{(}\bigg{[}\fracddtz\pi_t(X),Y\bigg{]}\bigg{)}+\pi_0^\perp\bigg{(}\bigg{[}X,\fracddtz\pi_t(Y)\bigg{]}\bigg{)},
%	\end{align*}
%	which is exactly the cocycle condition $d_{\nabla^F}(\sigma)=0$.
%\end{proof}
The element $\sigma_0\in \Omega^1(\cali{F},\nu(\cali{F}))$ and $[\sigma_0]\in H^1(\cali{F};\nu(\cali{F}))$ are called the \emph{deformation cocycle} and the \emph{deformation cohomology class} associated to a smooth deformation $(F_t)_{t\in I}=(T\cali{F}_t)_{t\in I}$ of a foliation $F=T\cali{F}$.
\begin{defin}
Two smooth deformations $(F_t)_{t\in I}$ and $(F'_t)_{t\in I}$ of a foliation $F\subset TM$ are said to be \emph{equivalent} if there exists an isotopy\footnote{Recall, by Definition \ref{defin of isotopy}, that an isotopy $(\phi_t)_{t\in I}$ is a smooth family of diffeomorphisms of $M$ such that $\phi_0=\id$.} $(\phi_t)_{t\in I}$ of $M$ such that $(\phi_t)_*(F_t)=F'_t$, for each $t\in I$. A smooth deformation $(F_t)_{t\in I}$ is called \emph{constant} if $F_t=F$ for all $t$. A smooth deformation is called \emph{trivial} if it is equivalent to the constant deformation.
\end{defin}
The following remark is based on \cite[Remark 4.1.]{delHoyo-Fernades-deformations-of-compact-foliations-2019}.
\begin{rem}\label{remark about the smooth vanishing condition for Moser thm for foliations}
There is an associated Bott connection $\nabla^{F_t}:\Gamma(F_t)\times\Gamma(\nu(\cali{F}_t))\to\Gamma(\nu(\cali{F}_t))$ to each $F_t$, defined by
\begin{equation*}
	\nabla^{F_t}_{X}(\pi_t^\perp(Y))=\pi_t^\perp([X,Y]), \quad \text{for any} \ X\in\Gamma(F_t) \ \text{and} \ Y\in\got{X}(M).
\end{equation*}
Denote by $\nu(\widetilde{\cali{F}})\coloneqq (TM\times TI)/\ttilde{F}$ the normal bundle to the total foliation $\ttilde{{F}}$ and by $\ttilde{\pi}:TM\times TI\to\nu(\widetilde{\cali{F}})$ the canonical projection. Analogously, there is a Bott connection $\nabla^{\widetilde{{F}}}:\Gamma(\ttilde{F})\times\Gamma(\nu(\widetilde{\cali{F}}))\to\Gamma(\nu(\widetilde{\cali{F}}))$ defined by
\begin{equation*}
	\nabla^{\ttilde{F}}_{\ttilde{X}}(\ttilde{\pi}(\ttilde{Y}))=\ttilde{\pi}([\ttilde{X},\ttilde{Y}]), \quad \text{for any} \ \ttilde{X}\in\Gamma(\ttilde{{F}}) \ \text{and} \ \ttilde{Y}\in\Gamma(TM\times TI).
\end{equation*}
Denote by $K\coloneqq (TM\times 0_I)/\ttilde{F}$ the subbundle of $\nu(\widetilde{\cali{F}})$ which is pointwise, for any $(p,t)\in M\times I$, is given by $K_{(p,t)}=T_pM/F_t=(\nu(\cali{F}_t))_p$. The Bott connection $\nabla^{\ttilde{F}}$ preserves the subbundle $K$ since $$\ttilde{\pi}([\ttilde{X},Y])\in\Gamma(K), \quad \ \text{for any} \ \ttilde{X}\in\Gamma(\ttilde{F}) \ \text{and} \ Y\in\Gamma(TM\times 0_I).$$ Therefore, we get the complex $(\Omega^\bullet(\widetilde{\cali{F}}; K), d_{\nabla^{\widetilde{{F}}}})$. The Bott connections $\nabla^{F_t}$ on $\nu(\cali{F}_t)$ and $\nabla^{\ttilde{F}}$ on $K$ are related, for any $X\in\Gamma(F)$ and $Y\in\Gamma(TM)$, by the following equation:
\begin{equation}\label{equation in MoserThmForFoliations relating global Bott connection with those at each time t}
	\nabla^{\ttilde{F}}_{\ttilde{X}_{(p,t)}}(\ttilde{\pi}(Y))=\ttilde{\pi}([\ttilde{{X}},Y]_{(p,t)})=\pi_t^\perp([X_t,Y_t]_p)=\nabla_{X_t|_p}^{F_t}\pi_t^\perp(Y_t),
\end{equation}
where $\ttilde{X}_{(p,t)}=(X_t|_p,0_t)$ and $Y_{(p,t)}\eqdef(Y_t|_p, 0_t)$ for any $(p,t)\in M\times I$. For each time $t\in I$, there is an associated deformation cocycle $\sigma_t\in\Omega^1(\cali{F}_t; \nu(\cali{F}_t))$ given by
\begin{equation*}
	\sigma_t(X)=\frac{d}{dt}\pi_t(X), \quad   X\in\Gamma(F_t).
\end{equation*}
These cocycles give rise to a global cocycle $\sigma\in \Omega^1(\widetilde{\cali{F}}; K)$ via the following equation:
\begin{equation}\label{equation which defines the global cocycle from those at each time t in moser thm for foliations}
	\sigma(\ttilde{X}_{(p,t)})\coloneqq\sigma_t(X_t|_p), \quad \text{for any} \ (p,t)\in M\times I.
\end{equation}
Notice that $\sigma$ is indeed a cocycle since $\{\sigma_t\}_{t\in I}$ are cocycles. More specifically, this is induced by the following formula
\begin{equation}\label{equation for the global differential and those at each time for the moser thm of foliations}
	d_{\nabla^{\ttilde{F}}}(\sigma)(\ttilde{X},\ttilde{Y})=d_{\nabla^{F_t}}(\sigma_t)(X_t,Y_t), \quad \ttilde{X}, \ttilde{Y}\in\Gamma(\widetilde{F}).
\end{equation}
\emph{In the following theorem the condition that the cocycles $\{\sigma_t\}_{t\in I}$ vanish smoothly in cohomology with respect to $t$--or equivalently that $[\sigma_t]\in H^1(\cali{F}_t;\nu(\cali{F}_t))$ vanish smoothly with respect to $t$--means that the cocycle $\sigma\in \Omega^1(\widetilde{\cali{F}}; K)$ is a coboundary, which means there exists $Y\in\Gamma(TM\times 0_I)$ such that
	\begin{equation}\label{equation for the smooth vanishing condition in moser thm for foliations}
		\sigma(\widetilde{X})=d_{\nabla^{\ttilde{F}}}(\ttilde{\pi}(Y))(\ttilde{X}), \quad \text{for any} \ \ \ttilde{X}\in\Gamma(\ttilde{F}).
\end{equation}}
Using \eqref{equation in MoserThmForFoliations relating global Bott connection with those at each time t}, \eqref{equation which defines the global cocycle from those at each time t in moser thm for foliations} and \eqref{equation for the global differential and those at each time for the moser thm of foliations}, we have for each $t\in I$ that:
\begin{equation*}
	\sigma(\ttilde{X})=d_{\nabla^{\ttilde{F}}}(\ttilde{\pi}(Y))(\ttilde{X})\Rightarrow \sigma_t(X_t)=d_{\nabla^{F_t}}(\pi_t^\perp(Y_t))(X_t)=\nabla^{F_t}_{X_t}\pi_t^\perp(Y_t)=\pi_t^\perp([X_t,Y_t]).
\end{equation*}
\end{rem}
Now, we are ready to give a direct proof of the Moser trick for foliations \cite[Theorem 2.]{Crainic-Moerdijk-Deformations-08'}.
\begin{thm}\label{Moser Theorem for foliations}
Let $(F_t)_{t\in I}$ be a smooth deformation of a foliation $F\subset TM$. If $(F_t)_{t\in I}$ is a trivial deformation then the deformation cocycles,
\begin{equation*}
	\sigma_t\eqdef\frac{d}{dt}\pi_t\in\Omega^1(\cali{F}_t;\nu(\cali{F}_t)),
\end{equation*}
vanish smoothly in cohomology with respect to $t$ (see Remark \ref{remark about the smooth vanishing condition for Moser thm for foliations} for the terminology). Conversely, if $M$ is compact and $[\sigma_t]\in H^1(\cali{F}_t; \nu(\cali{F}_t))$ vanish smoothly with respect to $t$, then $(F_t)_{t\in I}$ is trivial.
\end{thm}

\begin{proof}
Since $(F_t)_{t\in I}$ is a trivial deformation of $F$, there exists an isotopy $\phi_t:M\to M$ such that $(\phi_t)_*(F)=F_t$. Consider the vector field $\ttilde{Y}=Y_t+\partial_t$ on $M\times I$, where $Y_t\coloneqq\frac{d}{dt}\phi_t$ the time-dependent vector field generated by the isotopy $\phi_t$. The first claim is that the subbundle $D\subset TM\times TI$ given by
\begin{equation}\label{subbundle D in the proof of Moser thm for foliations}
	D_{(p,t)}\coloneqq \ttilde{F}_{(p,t)}\oplus \langle \ttilde{Y}_{(p,t)}\rangle, \quad \text{for any} \ (p,t)\in M\times I,
\end{equation}
is involutive and the second that this induces (is actually equivalent with) equation \eqref{equation for the smooth vanishing condition in moser thm for foliations}. First, consider the diffeomorphism $\ttilde{\phi}:M\times I\to M\times I$ defined by $\ttilde{\phi}(p,t)=(\phi_t(p), t)$. Now, we want to check that $(\ttilde{\phi})_*(F\oplus\langle\partial_t\rangle)=D$. We only need to compute how the Jacobian matrix of $\ttilde{\phi}$, in a local chart $\bigleftpar U, (x_1,\dots,x_n,t)\bigrightpar\subset M\times I$, acts on the local frame $(\partial_{x_1},\dots,\partial_{x_n},\partial_t)$.
\begin{equation}\label{the computation of a pushforward in the moser theorem for foliations}
	\left[
	\begin{array}{cccc|c}
		\frac{\partial \ttilde{\phi}_1}{\partial x_1} & \frac{\partial \ttilde{\phi}_1}{\partial x_2} & \cdots & \frac{\partial \ttilde{\phi}_1}{\partial x_n} & \frac{\partial \ttilde{\phi}_1}{\partial t} \\
		\frac{\partial \widetilde{\phi}_2}{\partial x_1} & \frac{\partial \ttilde{\phi}_2}{\partial x_2} & \cdots & \frac{\partial \ttilde{\phi}_2}{\partial x_n} & \frac{\partial \ttilde{\phi}_2}{\partial t} \\
		\vdots & \vdots & \ddots & \vdots & \vdots \\
		\frac{\partial \widetilde{\phi}_n}{\partial x_1} & \frac{\partial \ttilde{\phi}_n}{\partial x_2} & \cdots & \frac{\partial \ttilde{\phi}_n}{\partial x_n} & \frac{\partial \ttilde{\phi}_n}{\partial t} \\
		\hline
		0 & 0 & \cdots & 0 & 1
	\end{array}
	\right]
	\begin{array}{l}
		\left[
		\begin{array}{c}
			\partial_{x_1} \\
			\partial_{x_2} \\
			\vdots \\
			\partial_{x_n} \\
			\cline{1-1}
			\partial_t
		\end{array}
		\right]
	\end{array}
	=
	\begin{bmatrix}
		\sum\limits_{j=1}^n \frac{\partial \ttilde{\phi}_1}{\partial x_j} \partial_{x_j} \\
		\sum\limits_{j=1}^n \frac{\partial \ttilde{\phi}_2}{\partial x_j} \partial_{x_j} \\
		\vdots \\
		\sum\limits_{j=1}^n \frac{\partial \ttilde{\phi}_n}{\partial x_j} \partial_{x_j} \\
		0
	\end{bmatrix}
	+
	\begin{bmatrix}
		\frac{\partial \ttilde{\phi}_1}{\partial t} \partial_t \\
		\frac{\partial \ttilde{\phi}_2}{\partial t} \partial_t \\
		\vdots \\
		\frac{\partial \ttilde{\phi}_n}{\partial t} \partial_t \\
		0
	\end{bmatrix}
	+
	\begin{bmatrix}
		0 \\
		0 \\
		\vdots \\
		0 \\
		\partial_t
	\end{bmatrix}.
\end{equation}
Hence, we obtain that indeed the equation $(\ttilde{\phi})_*(F\oplus\langle\partial_t\rangle)=(\phi_t)_*(F)+\langle\ttilde{Y}\rangle=D$ holds. Since the subbundle $F\oplus\langle\partial_t\rangle$ of $TM\times TI$ is involutive, as a product of two involutive ones: $F\times\{0\}$ and $\{0\}\times\langle\partial_t\rangle$, and moreover $\ttilde{\phi}$ is a diffeomorphism, the involutivity of $D$ follows directly. Therefore, this means $[\ttilde{X},\ttilde{Y}]\in\Gamma(D)$ for any $\widetilde{X}\in\Gamma(\widetilde{F})$, which induces, for each $t\in I$, that
\begin{equation}\label{equation for the involutivity of subbundle D in deformations of foliations}
	[X_t,Y_t+\partial_t]\in \Gamma(D), \quad \text{where} \ \ttilde{X}=(X_t, 0_t)\in\Gamma(\ttilde{F}).
\end{equation}
First, notice that the vector field $[X_t,Y_t]$ has no $\partial_t$-component, due to the involutivity of $F_t$. Furthermore, the vector field $[X_t,\partial_t]$ also has no $\partial_t$-component. This can be observed by writing $\ttilde{X}=f_i(p,t)e^i_t$ around a point $(p,t)\in M\times I$, where $\{e^i_t\}_{i=1}^k$ is a local frame of $\ttilde{{F}}$, and performing the following quick computation:
\begin{equation*}
	[f_i(p,t)e^i_t,\partial_t]=\frac{\partial f_i(p,t)}{\partial t}e^i_t+f_i(p,t)\frac{\partial e^i_t}{\partial t}\in\got{X}(M).
\end{equation*}
We conclude that $[\ttilde{X},\ttilde{Y}]\in\Gamma(\ttilde{F})$ and so $\ttilde{\pi}([\ttilde{X},\ttilde{Y}])=0$. This is translated, for each $t\in I$, by the equation: $\pi_t^\perp([X_t,Y_t+\partial_t])=0$. Using the fact that $\pi_t(X_t)=X_t$ and $\pi_t^\perp(\frac{d}{dt}\pi_t(X_t))=\frac{d}{dt}\pi_t(X_t)$ (see Lemma \ref{lemma for differentiating the orthogonal projections in deformations of foliations}), we obtain that
\begin{equation}\label{equation that is used to prove that a trivial deformation of a fol gives rise to exact}
	\pi_t^\perp([X_t,Y_t])=\pi_t^\perp([\partial_t, X_t])= \pi_t^\perp([\partial_t,\pi_t(X_t)])=\pi_t^\perp\bigg{(}\frac{d}{dt}\pi_t(X_t)\bigg{)}=\sigma_t(X_t).
\end{equation}
(Notice that in the last part of the proof we showed that the involutivity of the subbundle $D$ \eqref{equation for the involutivity of subbundle D in deformations of foliations} induces \eqref{equation that is used to prove that a trivial deformation of a fol gives rise to exact} but it is immediate that the same proof can be simply reversed and let us conclude that the involutivity of $D$ and \eqref{equation for the smooth vanishing condition in moser thm for foliations} are equivalent.)\\

Conversely, assume there exists $Y\in\Gamma(TM\times 0_I)$ such that $\sigma(\ttilde{X})=d_{\nabla^{\ttilde{F}}}(\ttilde{\pi}(Y))(\ttilde{X})$ for any $\ttilde{X}\in\Gamma(\ttilde{{F}})$. Equivalently, there exists a time-dependent vector field $Y=\{Y_t\}_{t\in I}$ on $M$ such that $\sigma_t(X_t)=\pi_t^\perp([X_t,Y_t])$ for each $t\in I$. We consider again the vector field $\ttilde{Y}=Y_t+\partial_t$ on $M\times I$. As we explained above, the involutivity of the subbundle $D\subset TM\times TI$ given by \eqref{subbundle D in the proof of Moser thm for foliations} is equivalent with the equation \eqref{equation for the smooth vanishing condition in moser thm for foliations}. Since the flow of vector fields tangent to a foliation preserve the leaves, we obtain that $(\phi_{\ttilde{Y}}^t)_*(D)=D$ where $\phi^t_{\ttilde{Y}}:M\times I\to M\times I$ is the flow of $\ttilde{Y}\in\Gamma(D)$ which is defined for all $t\in I$ because $M$ is compact. Recall that, by \eqref{relation between the time-dependent flow and the time-independent}, we have the following relation
\begin{equation}
	\phi_{\ttilde{Y}}^t(p,0)=(\Phi^t_Y(p),t), \quad \text{for any} \ (p,t)\in M\times I,
\end{equation}
where $\Phi_Y^t$ the time-dependent flow of $Y$. Hence, $\phi^t_{\ttilde{Y}}(M\times\{0\})=M\times\{t\}$ which induces that $(\phi^t_{\ttilde{Y}})_*(D|_{{M}\times\{0\}})=D|_{M\times\{t\}}$. Moreover, recall that $D|_{M\times\{t\}}=F_t\oplus\langle Y_t+\partial_t\rangle.$ Now, a completely analogous computation to \eqref{the computation of a pushforward in the moser theorem for foliations}, together with the fact that $(\phi^t_{\ttilde{Y}})_*(Y_0|_{M\times\{0\}})=Y_t|_{M\times\{t\}}$, shows that $$(\phi^t_{\ttilde{Y}})_*(D|_{M\times\{0\}})=(\phi^t_{\ttilde{Y}})_*(F\oplus\langle Y_0+\partial_t|_0\rangle)=(\Phi_Y^t)_*(F)\oplus\langle Y_t+\partial_t \rangle,$$ and thus $\Phi^t_Y$ is the desired isotopy trivializing the deformation $(F_t)_{t\in I}$.
\end{proof}
Note that, in the following lemma, no compactness assumption on the ambient manifold $M$, nor the existence of a complete vector field on it, is imposed. Thus, it recovers the result of \cite[Proposition 2.12]{Heitsch-A-cohomology-for-foliated-manifolds-1975} under a substantially weaker assumption.
\begin{lem}
A trivial deformation $(F_t)_{t\in I}$ of a foliation $F$ gives rise to an exact deformation cocycle $\sigma_0\in\Omega^1(\cali{F}, \nu(\cali{F}))$.
\end{lem}
\begin{proof}
We need to check that there exists $Y\in\got{X}(M)$ such that $\sigma_0(X)=\pi_0^\perp[X,Y]$, for any $X\in\Gamma(F)$. This follows directly from the proof of Theorem \ref{Moser Theorem for foliations}. We consider $t=0$ in equation \eqref{equation that is used to prove that a trivial deformation of a fol gives rise to exact}.
\end{proof}

\begin{rem}
The proof of the Moser trick for Lie algebroids \cite[Theorem~2.]{Crainic-Moerdijk-Deformations-08'} shares many similarities with the proof that gauge equivalences of Lie algebroids correspond to their geometric equivalences \cite[Proposition~1.1.5.]{LaPastina-Vitagliano-deformations-VB-algebroids-2019}. Analogously, Theorem~\ref{Moser Theorem for foliations} is closely related to \cite[Theorem~2.6]{Zambon-Schaetz-Gaugeequivalences-for-foliations-2021}. It would be interesting to clarify the precise connection between the deformation-theoretic Moser trick and the equivalence between gauge and geometric equivalences. This more subtle link is left for future investigation.
\end{rem}

\appendix

	\section{On curves in the Grassmannian}\label{appendix on curves in Grassmannian}
Let $G$ be a finite dimensional real Lie group. Spaces which carry a transitive Lie group $G$-action are called homogeneous $G$-spaces. If $V$ is a finite dimensional real vector space, its Grassmannian of $k$-dimensional vector subspaces $\Gr_k(V)\coloneqq\{W\subset V : \dim W=k\}$ is a homogeneous $\GL(V)$-space and the action is given just by the matrix multiplication (upon choosing a basis for $V$). Transitivity can be seen in the following way: given bases for two subspaces $W_1$ and $W_2$ of $V$, we can extend them to bases for $V$ and then the linear map which sends one basis to the other, sends also $W_1$ to $W_2$. By the theory of homogeneous spaces, we know that these spaces are quotients of Lie groups modulo closed subgroups. Specifically, for a homogeneous $G$-space $X$, let $x\in X$, we have the following diffeomorphism
\begin{equation*}
	G/G_x\simeq X, \quad \overline{g}\mto gx
\end{equation*}
where $G_x=\{g\in G| \ gx=x\}$ is the isotropy group at the point $x$. Hence, translating the above diffeomorphism in the case of Grassmannians we obtain that $$\Gr_k(V)\simeq \GL(V)/{\GL(V)_W}$$ for any $W\in \Gr_k(V)$ and $\GL(V)_W$ the isotropy group of $W$. Applying the tangent functor to the above diffeomorphism, we get the following vector space isomorphism
\begin{equation*}\label{isomorphism for tangent of Grassmannian}
	T_W(\Gr_k(V))\simeq T_{[\Id]}(\GL(V)/\GL(V)_W.
\end{equation*}
We can make the above isomorphism a bit more specific:
\begin{equation}\label{vector space isomorphism ofr Grassmannians}
	T_{[\Id]}(\GL(V)/\GL(W; V))\simeq \got{gl}(V)/\got{gl}(V)_W
\end{equation}
where $\got{gl}(V)_W=\{B\in\got{gl}(V)| \ BW=W\}$ is the Lie algebra corresponding to $\GL(V,W)$. Denote by $\Pi:\GL(V)\to\GL(V)/\GL(V, W)$ the canonical projection and consider its tangent map at the identity:$$T_{\Id}\Pi:\got{gl}(V)\to T_{[\Id]}\bigleftpar\GL(V)/\GL(V)_W\bigrightpar.$$Hence, $T_{\Id}(\Pi^{-1}(\overline{0}))=T_{\Id}\GL(V)_W=\ker(T_{\Id}\Pi)=\got{gl}(V)_W$.

By the first isomorphism theorem, we obtain \eqref{vector space isomorphism ofr Grassmannians}. 
Applying, once again, the first isomorphism theorem to the map $\pi_{V/W}\circ \res:\Hom(V, V)\to\Hom(W, V/W)$, where $\pi_{V/W}\circ\res$ is just the restriction map $\res:\Hom(V, V)\to\Hom(W, V)$ followed by the canonical projection $\pi_{V/W}:V\to V/W$, we get that $\got{gl}(V)/\got{gl}(V)_W\simeq W^*\otimes V/W$. Summing up, we conclude that
\begin{equation}\label{canonical identification of the tangent space of the Grassmannian}
	T_W(\Gr_k(V))\simeq W^*\otimes V/W.
\end{equation}
It will be useful, for our purposes, to interpret the above isomorphism in terms of curves. The isotropy group $G_x$ at any $x\in X$ is a closed subgroup of $G$ and acts on $G$ just by the group multiplication. By quotient manifold theorem (\cite{Lee-Smooth-manifolds-2013}, Thm. 21.10), we can conclude that $G\stackrel{\Pi}{\to} G/G_x$ is a $(G_x)$-principal bundle. Into the context of Grassmannians, we get that the canonical projection$$\GL(V)\stackrel{\Pi}{\to}\GL(V)/\GL(V)_W$$is a $\GL(V)_W$-principal bundle. Let $J\subset\mbb{R}$ be an open interval containing the origin, and $\gamma:J\to \Gr_k(V)$ be a smooth curve in $\Gr_k(V)$, with $\gamma(0)=W$. We denote its value $\gamma(t)$ by $W_t$.
\begin{equation}\label{diagram for the existence of a curve of isomorphisms sending initial subalgebra to its deformation}
	\begin{tikzcd}
		& \GL(V) \arrow[d, "\Pi"]\\
		J \ar[ru, dashed, "\alpha"] \arrow[r, swap, "\gamma"]& \Gr_k(V)\simeq \GL(V)/{\GL(V)_W} 
	\end{tikzcd}
\end{equation}
By local triviality of principal bundles, there exists an open subinterval $I\subseteq J$ containing zero, and a curve $\alpha:I\to \GL(V)$, denoted by $\alpha(t)\eqqcolon \alpha_t$, such that $\Pi\circ\alpha_t=W_t$ for each $t\in I$. Equivalently, the last condition means that $\alpha_0=\Id$ and $\alpha_t(W)=W_t$ for each $t\in I$ . Differentiating at $t=0$, we get the equation $\fracddtzs{W_t}=T_{\Id}\Pi\circ\fracddtzs{\alpha_t}$. Using the isomorphism \eqref{vector space isomorphism ofr Grassmannians}, we conclude that  the interpretation of \eqref{canonical identification of the tangent space of the Grassmannian} via curves is given by:$$\fracddtz{W_t}\mto \pi_{V/W}\circ\fracddtz{\alpha_t}.$$
We have just explained in detail that, given $W\in\Gr_k(\mathfrak{g})$ and a smooth curve $(W_t)_{t\in J}\subset\Gr_k(\mathfrak{g})$ with $W_0=W$, there exists a smooth curve $(\alpha_t)_{t\in I}\subset\mathrm{GL}(\mathfrak{g})$, where $0\in I\subseteq J$, such that $\alpha_0=\Id$ and $\alpha_t(W)=W_t$, for every $t\in I$.
Therefore, for each $t\in I$, we obtain the following commutative diagram.
\begin{equation}\label{commutative diagram in the proof of the cocycle in Ad-submodules}
	\begin{tikzcd}
		\got{g}\arrow[r,"\alpha_t"] \arrow[d,swap,"\pi_{\got{g}/\tiny W}"] &
		\got{g} \arrow[d,"\pi_{\got{g}/W_t}"]\\
		\got{g}/W \arrow[r,"\overline{\alpha_t}"] & \got{g}/{W_t}
	\end{tikzcd},
\end{equation}
where the vertical arrows denote the corresponding canonical projections.\\

Let $(\got{h},\mu_{\got{h}})\subset (\got{g},\mu_{\got{g}})$ be a Lie subalgebra, and let $\iota_{\got{h}}:\got{h}\hookrightarrow \got{g}$ denote its natural inclusion. Consider a smooth curve $(\got{h}_t)_{t\in J}\subset \Gr_k(\got{g})$ such that $\got{h}_0=\got{h}$ and, for every $t\in J$, $\got{h}_t$ is a Lie subalgebra of $\got{g}$. For each $t\in J$, denote by $\mu_{\got{h}_t}\coloneqq \mu_{\got{g}}|_{\wedge^2\got{h}_t}$ the restriction of the ambient Lie bracket to $\got{h}_t$, and let $\iota_{\got{h}_t}:(\got{h}_t,\mu_{\got{h}_t})\hookrightarrow (\got{g},\mu_{\got{g}})$ be the corresponding natural inclusions.

From our previous discussion, possibly after shrinking to an interval $I\subseteq J$, $\got{h}_t=\alpha_t(\got{h})$ for each $t\in I$, for some $(\alpha_t)_{t\in I}\subset\GL(\got{g})$ with $\alpha_0=\id_{\got{g}}$. The following formula defines a smooth deformation of the Lie bracket $\mu_{\got{h}}$:
\begin{equation}\label{induced deformation of the lie bracket of the lie algebra}
	\mu_{\got{h}}^t(u,v)\eqdef \alpha_t^{-1}\circ \mu_{\got{g}}\bigleftpar \alpha_t(u),\alpha_t(v)\bigrightpar, \quad \text{for all } t\in I \text{ and } u,v\in\got{h},
\end{equation}
such that $\alpha_t^{-1}:(\got{h}_t,\mu_{\got{h}_t})\stackrel{\simeq}{\to}(\got{h},\mu_{\got{h}}^t)$ is a Lie algebra isomorphism for every $t\in I$. This construction naturally comes with a smooth family of injective Lie algebra homomorphisms $\iota_t:(\got{h},\mu_{\got{h}}^t)\hookrightarrow(\got{g},\mu_{\got{g}})$ defined by $\iota_t\eqdef\iota_{\got{h}_t}\circ\alpha_t|_{\got{h}}$. The Grassmannian point of view is recovered by $\bigleftpar\text{Im}(\iota_t)\bigrightpar_{t\in I}\subset\Gr_k(\got{g})$.\\

In the special case where $\got{h}\subset\got{g}$ is a Lie ideal, there is, in addition, a smooth family of surjective Lie algebra homomorphisms$$\pi_t:(\got{g},\mu_{\got{g}})\twoheadrightarrow(\got{g}/\got{h},\mu^t_{\got{g}/\got{h}}),$$ defined by $\pi_t\eqdef\overline{\alpha_t^{-1}}\circ\pi_{\got{g}/\got{h}_t}\stackrel{(\ref{diagram for the existence of a curve of isomorphisms sending initial subalgebra to its deformation})}{=}\pi_{\got{g}/\got{h}}\circ\alpha_t^{-1}$, where
\begin{equation*}
		\mu_{\got{g}/\got{h}}^t(\overline{x},\overline{y})\eqdef\overline{\alpha_t^{-1}}\circ \mu_{\got{g}}\bigleftpar \overline{\alpha_t}(x),\overline{\alpha_t}(y)\bigrightpar, \quad \text{for all } t\in I \text{ and } \overline{x},\overline{y}\in\got{g}/\got{h}.
\end{equation*}
Thus, we obtain for each $t\in I$ the following commutative diagram of short exact sequences of Lie algebras.
\[
\begin{tikzcd}
	0 \arrow[r] &
	(\got{h}_t,\mu_{\got{h}_t}) \arrow[r, "\iota_{\got{h}_t}"] \arrow[d, "\alpha^{-1}_t"] \arrow[d, swap, "\rotatebox{-90}{$\simeq$}"] &
	\got{g} \arrow[r, "\pi_{\got{g}/\got{h}_t}"] \arrow[d, "\Id_{\got{g}}"] &
	(\got{g}/{\got{h}_t},\mu_{\got{g}/\got{h}_t}) \arrow[r] \arrow[d, swap,
	"\overline{\alpha^{-1}_t}"'] \arrow[d, swap, "\rotatebox{-90}{$\simeq$}"]
	&
	0 \\
	0 \arrow[r] &
	(\got{h},\mu_{\got{h}}^t) \arrow[r, "\iota_t"] &
	\got{g} \arrow[r, "\pi_t"] &
	(\got{g}/{\got{h}},\mu_{\got{g}/\got{h}}^t) \arrow[r] &
	0.
\end{tikzcd}
\]
For an ideal, the Grassmannian point of view is equivalently described by the smooth family $\bigl(\ker(\pi_t)\bigr)_{t\in I}\subset\Gr_k(\got{g})$.\\

\textbf{Conclusion}: As is common in deformation theory, the present manuscript is concerned not with arbitrarily large deformations, but rather with sufficiently small ones. This appendix therefore shows that the definitions of deformations of Lie subalgebras and Lie ideals in terms of curves in the Grassmannian are locally equivalent to definitions (C) and (D) introduced in Section \ref{second section of the paper deformations of foliations}, respectively, and are therefore sufficient for our purposes.

\section{On Time-Dependent Vector Fields}\label{appendix}

We briefly recall a few basics on time-dependent vector fields, which can be found, e.g., in \cite{Jafarpour-Lewis-book-Time-varying-vector-fields-and-their-flows-2014}.

A time-dependent vector field is a smooth family of vector fields ${X}=\{X_t\}_{t\in I}$, where $I\subset\mathbb{R}$ is an open interval containing the origin. Specifically, it is a smooth map $${X}:I\times M\to TM, \quad {X}(t,p)\eqqcolon X_t(p)\in T_pM.$$

The flow $\Phi^{t,s}_{{X}}$ of a time-dependent vector field ${X}$, consists of local diffeomorphisms $\Phi^{t,s}_{{X}}:M\to M$, which depend both on the parameters $t$ and $s$, determined by
\begin{equation}\label{defining relations for the time-dependent flow}
	\frac{d}{dt}\Phi^{t,s}_{{X}}(p)=X_t(\Phi^{t,s}_{{X}}(p)), \quad \Phi^{s,s}_{{X}}(p)=p.
\end{equation}
Moreover,  it satisfies the property $\Phi^{t,u}_{{X}}\circ\Phi^{u,s}_{{X}}=\Phi^{t,s}_{{X}}$ by taking care of the values of the flow domain analogously to the time-independent case. The time-dependent vector field ${X}=\{X_t\}$, can be seen as time-independent vector field $\widetilde{X}$ on $M\times I$ as follows $$\widetilde{X}(t,p)=X_t(p)+\partial_t.$$ The flow of $\widetilde{X}$ and the (time-dependent) flow of $X=\{X_t\}$ are related by the following relation
\begin{equation}\label{relation between the time-dependent flow and the time-independent}
	\phi_{\widetilde{X}}^t(p,s)=(\Phi^{t+s,s}_{X}(p), t+s).
\end{equation}
In the context of deformations, we are mainly interested in what happens for parameters around $0\in I$ and so it is enough to study the family of local diffeomorphisms $$\Phi^t_{{X}}\coloneqq\Phi^{t,0}_{{X}}.$$
\begin{defin}\label{defin of isotopy}
A smooth one parameter family of (global) diffeomorphisms $(\phi_t)_{t\in I}\in\text{Diff}(M)$ such that $\phi_0=\id$, is called an \emph{isotopy} of $M$.
\end{defin}
Every isotopy $\phi_t:M\to M$ gives rise to a time-dependent vector field ${Y}=\{Y_t\}_{t\in I}$ defined, for each $p\in M$, by
\begin{equation}\label{time-dependent vector field associated to an isotopy}
Y_t(\phi_t(p))=\frac{d}{dt}\phi_t(p)
\end{equation}
For the sake of simplicity, we use the notation $Y_t=\frac{d}{dt}\phi_t$. Its corresponding (time-dependent) flow is given by
\begin{equation}\label{time-dependent flow associated to an isotopy}
\Phi_{{Y}}:I\times I\times M\to M, \quad \Phi_{{Y}}^{t,s}(p)=\phi_t\circ\phi_{s}^{-1}(p),
\end{equation}
which directly can be checked by the relations in \eqref{defining relations for the time-dependent flow}.\\

\bibliographystyle{plain} % or your preferred style
\bibliography{Mosertrick} % replace with your .bib filename

\end{document}